\documentclass[12pt]{amsart}
\usepackage{amssymb}
\usepackage{amsthm}
\usepackage{amsmath}
\usepackage{enumerate}
\usepackage[colorlinks]{hyperref}
\usepackage{mathrsfs}
\usepackage{booktabs}
\usepackage{marginnote}

\theoremstyle{plain}

\newtheorem{theorem}{Theorem}[section]
\newtheorem{lemma}[theorem]{Lemma}

\newtheorem{corollary}[theorem]{Corollary}

\theoremstyle{remark}
\newtheorem*{remark}{Remark}

\newcommand{\GL}{{\mathrm {GL}}}
\newcommand{\PGL}{{\mathrm {PGL}}}
\newcommand{\SL}{{\mathrm {SL}}}
\newcommand{\PSL}{{\mathrm {PSL}}}

\newcommand{\GU}{{\mathrm {GU}}}

\newcommand{\PSU}{{\mathrm {PSU}}}

\newcommand{\PO}{{\mathrm {P}\Omega}}

\newcommand{\Sp}{{\mathrm {Sp}}}
\newcommand{\PSp}{{\mathrm {PSp}}}

\newcommand{\Aut}{{\mathrm {Aut}}}

\newcommand{\Irr}{{\mathrm {Irr}}}

\newcommand{\St}{{\mathrm {St}}}
\newcommand{\Stab}{{\mathrm {Stab}}}

\newcommand{\Gal}{{\it Gal}}

\newcommand{\NN}{{\mathbb N}}

\newcommand{\FF}{{\mathbb F}}

\newcommand{\bZ}{\mathbf{Z}}

\newcommand{\bC}{{\mathbf{C}}}
\newcommand{\bN}{\mathbf{N}}
\newcommand{\bF}{\mathbf{F}}

\newcommand{\Al}{\textup{\textsf{A}}}
\newcommand{\Sy}{\textup{\textsf{S}}}

\begin{document}
\title[Finite groups with a character of large degree]
{Finite groups with an irreducible\\ character of large degree}

\thanks{Nguyen N. Hung is partially supported by the NSA Young
  Investigator Grant \#H98230-14-1-0293 and a Faculty Scholarship
  Award from the Buchtel College of Arts and Sciences, The University
  of Akron}

\author[N.\,N. Hung]{Nguyen Ngoc Hung}
\address{Department of Mathematics, The University of Akron, Akron,
Ohio 44325, USA} \email{hungnguyen@uakron.edu}

\author[M.\,L. Lewis]{Mark L. Lewis}
\address{Department of Mathematical Sciences, Kent State University, Kent,
OH 44242, USA} \email{lewis@math.kent.edu}

\author[A.\,A. Schaeffer Fry]{Amanda A. Schaeffer Fry}
\address{Department of Mathematical and Computer Sciences, Metropolitan State University of Denver, Denver, CO 80217, USA} \email{aschaef6@msudenver.edu}



\date{\today}

\begin{abstract} Let $G$ be a finite group and $d$ the degree of a complex irreducible character of $G$, then write $|G|=d(d+e)$ where
$e$ is a nonnegative integer. We prove that $|G|\leq e^4-e^3$
whenever $e>1$. This bound is best possible and improves on several
earlier related results.
\end{abstract}

\maketitle


\section{Introduction}

Let $d$ be the degree of a complex irreducible character of a finite
group $G$. Since $d$ divides $|G|$ and $d^2 \leq |G|$, one can write
$|G| = d(d+e)$ for some nonnegative integer $e$. It is clear that
the largest possible value of $d$ is $\sqrt {|G|}$ and $d = \sqrt
{|G|}$ if and only if $G$ is trivial.

The extremal situations where $d$ is close to $\sqrt {|G|}$ or
equivalently $e$ is small have been studied considerably in the
literature. In \cite{Berkovich}, Y.~Berkovich showed that $e = 1$ if
and only if $G$ is either a cyclic group of order $2$ or a
$2$-transitive Frobenius group. Going further, N.~Snyder
\cite{Snyder} classified the finite groups with $e = 2$ or $3$, and
as a consequence of his classification, $|G| \leq 8$ when $e = 2$
and $|G| \leq 54$ when $e = 3$. This naturally leads Snyder to the
observation that $|G|$ is bounded in terms of $e$ whenever $e>1$
and, indeed, he managed to prove that $|G|\leq ((2e)!)^2$.

Finding the best bound for $|G|$ in terms of $e$ has become a
problem of interest in many recent papers. I.\,M.~Isaacs
\cite{Isaacs2} was the first to improve Snyder's factorial bound to
a polynomial one of the form $Be^6$ where $B$ is a large enough
constant. However his proof relied on a result of M.~Larsen,
G.~Malle, and P.\,H.~Tiep \cite[Theorem~1.1]{Larsen-Malle-Tiep} on
bounding the largest irreducible character degree in terms of
smaller degrees in a simple group, which in turn replies on the
classification of finite simple groups. Later on, C.~Durfee and
S.~Jensen \cite{Durfee-Jensen} were able to obtain the bound of $e^6
- e^4$ without using the classification. This bound was further
improved to $e^4 + e^3$ by the second author in \cite{Lewis}.

In \cite{Isaacs2}, Isaacs pointed out that the group of $3\times 3$
matrices of the form
\[ \left( \begin{array}{ccc}
1 & x & y \\
0 & 1 & z \\
0 & 0 & t \end{array} \right),\] where $x,y,z,t$ are elements in a
field of order $q$ and $t\neq 0$, has order $q^3(q-1)$ and an
irreducible character of degree $q(q-1)$. These groups show that the
best possible bound one can achieve is $e^4 - e^3$ and, in fact,
this bound holds when $G$ has a nontrivial abelian normal subgroup,
as shown in \cite[Theorem~1]{Lewis}. We note that these groups had
earlier appeared in \cite[p.~383]{Gagola} in a slightly different
context.

The aim of the present paper is to prove the optimal bound of $e^4 -
e^3$ for arbitrary finite groups.

\begin{theorem}\label{theorem-main-2} Let $|G| = d(d+e)$ where $e > 1$ and $d$ is the degree of some
irreducible character of $G$. Then $|G| \leq e^4 - e^3$.
\end{theorem}

In light of \cite{Lewis}, to prove Theorem~\ref{theorem-main-2} it
suffices to assume that $G$ has a trivial solvable radical. Indeed,
we can do a bit more.

\begin{theorem}\label{theorem-main} Let $|G| = d(d+e)$ where $d$ is the degree of some
irreducible character of $G$. If $G$ has a non-abelian minimal
normal subgroup, then $|G| < e^4 - e^3$.
\end{theorem}

Theorem~\ref{theorem-main} convinces us that those groups with
$|G|=e^4-e^3$ are necessarily solvable. It would be interesting to
confirm this, or to even classify them completely, a task that seems
nontrivial to us. In Section~\ref{section: final}, we show that they
must be the so-called Gagola groups of specific type and present
some of their examples.

Let $\bF(G)$ and $b(G)$ respectively denote the Fitting subgroup and
the largest degree of an irreducible character of $G$. An old (and
still open) conjecture of Gluck \cite{Gluck} asserts that
$|G:\bF(G)|\leq b(G)^2$ whenever $G$ is solvable. In a recent
extension of Gluck's conjecture to arbitrary finite groups
\cite{Cossey-Halasi-Maroti-Nguyen}, it has been predicted that
$|G:\bF(G)|\leq b(G)^3$. This means that, when $G$ has a trivial
solvable radical, it is expected that $|G|\leq b(G)^3$. In the
course of proving Theorem~\ref{theorem-main}, we in fact prove that
$e>\sqrt{b(G)}+1$, and this, on the other end, provides a
\emph{lower bound} for $|G|$ in terms of $b(G)$ in those groups.

\begin{theorem}\label{theorem-main-3}
Let $G$ be a finite group with a non-abelian minimal normal
subgroup. Then
\[|G| >b(G)(b(G)+\sqrt{b(G)}+1).\]
\end{theorem}

Theorem~\ref{theorem-main-3} is not true for non-solvable groups in
general, as shown by non-solvable $2$-transitive Frobenius groups.
We should also mention that we know of no finite groups $G$ with a
non-abelian minimal normal subgroup such that $|G|\leq 2b(G)^2$. In
fact, we are able to prove the following, which solves a weak form
of a prediction of Isaacs raised in \cite{Isaacs2}, see
Section~\ref{section:simple groups} for a detailed discussion.

\begin{theorem}\label{theorem-nonabelian-simple-groups}
Let $S$ be a finite non-abelian simple group. Then $|S|>2b(S)^2$.
Consequently, if $|S| = d(d+e)$ where $d$ is the degree of some
irreducible character of $S$ then $|S| < 2e^2$.
\end{theorem}

Our proof of Theorem~\ref{theorem-main} is fundamentally different
from those in \cite{Isaacs2,Durfee-Jensen, Lewis} and, as expected,
replies on the classification of finite simple groups. Let $N$ be a
non-abelian minimal normal subgroup of $G$ and suppose that $S$ is a
simple direct factor of $N$. The proof is divided in two main cases,
according to whether or not $S$ is isomorphic to $\PSL_2(q)$.

The key to the proof in the case $S\ncong \PSL_2(q)$ is to show that
$S$ possesses an irreducible character $\theta$ extendible to
$\Aut(S)$ of `very large' degree, namely $\theta(1) > |S|^{3/8}$,
see Theorem~\ref{theorem3/8}. This result, which might have other
applications, together with
Theorem~\ref{theorem-nonabelian-simple-groups} and recent results
\cite{Larsen-Malle-Tiep,Halasi-Hannusch-Nguyen} on bounding the
largest character degree in terms of smaller degrees in finite
simple groups allow us to obtain the desired bound, see
Section~\ref{section:simple groups} and
Theorem~\ref{theorem-S-ncongPSL(2,q)}. The case $S \cong \PSL_2(q)$
turns out to be surprisingly complicated and requires delicate
treatment, and is done in Sections~\ref{section:S=PSL even} and
\ref{section:S=PSLodd}.


\section{Extendible characters of simple groups}
\label{section:extendible characters}

In this section we will show that a non-abelian simple group
$S\ncong \PSL_2(q)$ has an irreducible character extendible to
$\Aut(S)$ of very large degree. The following theorem is a key tool
toward the proof of Theorem~\ref{theorem-main} in the case
$S\ncong\PSL_2(q)$.

\begin{theorem}\label{theorem3/8}
Let $S$ be a non-abelian simple group
not isomorphic to $\PSL_2(q)$ where $q$ is a prime power. Then $S$
has an irreducible character $\theta$ extendible to $\Aut(S)$ such
that $\theta(1)>|S|^{3/8}$.
\end{theorem}

\begin{remark}
The exclusion of $\PSL_2(q)$ in the theorem is necessary since
$|\PSL_2(q)|=q(q^2-1)/(2,q-1)$ and $b(\PSL_2(q))=q$ or $q+1$ for
$q\geq 5$.
\end{remark}

In the study of Gluck's conjecture
\cite{Cossey-Halasi-Maroti-Nguyen} concerning the largest character
degree and the index of the Fitting subgroup in a finite group, the
first author along with J.\,P.~Cossey, Z.~Halasi, and A.~Mar\'{o}ti
have proved that every non-abelian simple group $S$ possesses an
irreducible character extendible to $\Aut(S)$ with degree at least
$|S|^{1/3}$. Unfortunately this bound is not enough for our current
purpose. However, the ideas in the proof of
\cite[Theorem~12]{Cossey-Halasi-Maroti-Nguyen} can be further
developed to prove Theorem~\ref{theorem3/8}.

For the reader's convenience and to prove Theorem~\ref{theorem3/8}
for the alternating groups, we recall some combinatorics concerning
partitions, Young diagrams, and representation theory of the
alternating and symmetric groups.

Let $n$ be a positive integer. A finite sequence
$(\lambda_1,\lambda_2, \ldots ,\lambda_k)$ for some $k$ such that
$\lambda_1\geq\lambda_2\geq\cdots \geq\lambda_k$ and
$\lambda_1+\lambda_2+\cdots+\lambda_k=n$ is said to be a partition
of $n$. The Young diagram associated to $\lambda$, denoted by
$Y_\lambda$, is defined to be the finite subset of $\NN\times\NN$
such that $(i,j)\in Y_\lambda \text{ if and only if } i\leq
\lambda_j.$

When two Young diagrams that can be transformed into each other when
reflected about the line $y=x$, we say that the associated
partitions are conjugate. The partition conjugate to $\lambda$ is
denoted by $\overline{\lambda}$. If $\lambda=\overline{\lambda}$
then $Y_\lambda$ is symmetric and we say that $\lambda$ is
self-conjugate. For each node $(i,j)\in Y_\lambda$, we define the
so-called \emph{hook length} $h(i,j)$ to be the number of nodes that
are directly above it, directly to the right of it, or equal to it.
That is,
\[h(i,j):=1+\lambda_j+\overline{\lambda}_i-i-j.\]
It is well-known that there are bijective correspondences between
the partitions of $n$, the Young diagrams of cardinality $n$, and
the irreducible complex characters of $\Sy_n$. Denote by
$\chi_\lambda$ or $\chi_{Y_\lambda}$ the irreducible character of
$\Sy_n$ corresponding to $\lambda$ and $Y_\lambda$. The degree of
$\chi_\lambda$ is given by the \emph{hook-length formula} of
J.\,S.~Frame, G.\,B.~Robinson, and
R.\,M.~Thrall~\cite{Frame-Robinson-Thrall}:
\[\chi_{\lambda}(1)=\chi_{Y_\lambda}(1)=
\frac{n!}{\prod _{(i,j)\in Y_\lambda}h(i,j)}.\]

The irreducible characters of $\Al_n$ can be obtained by restricting
those of $\Sy_n$ to $\Al_n$. More explicitly,
$\chi_{\lambda}\downarrow_{\Al_n}=\chi_{\overline{\lambda}}\downarrow_{\Al_n}$
is irreducible of degree $\chi_{\lambda_1}(1)$ if $\lambda$ is not
self-conjugate. Otherwise, $\chi_{\lambda}\downarrow_{\Al_n}$ splits
into two different irreducible characters of the same degree
$\chi_{\lambda_1}(1)/2$.

Define $A(\lambda)$  to be the set of nodes that can be added to
$Y_\lambda$ to obtain another Young diagram of size $n+1$. It is
known (see~\cite[\S2]{Larsen-Malle-Tiep} for instance) that
\[|A(\lambda)|<\sqrt{2n}+1.\]
Similarly, define $R(\lambda)$ to be the set of nodes that can be
removed from $Y_\lambda$ to obtain another Young diagram of size
$n-1$. We have \[|R(\lambda)|<\sqrt{2n}.\]

The branching rule~\cite[\S9.2]{James} asserts that the restriction
$\chi_\lambda\downarrow_{\Sy_{n-1}}$ of $\chi_\lambda$ to
$\Sy_{n-1}$ is a sum of irreducible characters $\chi_{Y_\lambda
\backslash \{ (i,j) \}}$ as $(i,j)$ runs over all nodes in
$R(\lambda)$; and the induction $\chi_\lambda^{\Sy_{n+1}}$ of
$\chi_\lambda$ to $\Sy_{n+1}$ is a sum of irreducible characters
$\chi_{Y_\lambda \cup \{(i,j)\}}$ as $(i,j)$ runs over all nodes in
$A(\lambda)$.

\begin{proof}[Proof of Theorem~\ref{theorem3/8}]
If $S$ is a simple sporadic group or the Tits group, the proof is a
routine check from the Atlas~\cite{Atl1}. If $S$ is a simple group
of Lie type in characteristic $p$ and $S\ncong \PSL_2(q)$ where $q$
is a prime power, we then realize that $S$ has the so-called
\emph{Steinberg character} $\St_S$ of degree $\St_S(1)=|S|_p$, the
$p$-part of the order of $S$. Furthermore, $\St_S$ is extendible to
$\Aut(S)$ (see~\cite{Feit} for instance). Now we can check the
inequality $|S|_p>|S|^{3/8}$ easily by consulting the list of simple
groups and their orders, see~\cite[p. xvi]{Atl1} for instance.

So for the rest of this proof we assume that $S=\Al_n$ is an
alternating group of degree $n\geq7$. Note that $\Al_5\cong
\PSL_2(5)$ and $\Al_6\cong \PSL_2(9)$ are not in our consideration.
Let $\rho(\Al_n)$ be the largest degree of an irreducible character
of $\Al_n$ that can be extended to $\Sy_n$. We aim to show that
$\rho(\Al_n)>(n!/2)^{3/8}$ when $n\geq 7$.

Since the lemma can be checked directly by computer for small $n$,
we assume that $n\geq 75$. In fact, we will prove by induction on
$n\geq 75$ that $\rho(\Al_{n+1})\geq (n+1)^{3/8}\rho(\Al_n)$ and
this implies that $\rho(\Al_n)>(n!/2)^{3/8}$ immediately.

Let $\psi$ be an irreducible character of $\Al_n$ with $n\geq 75$
such that $\psi$ is extendible to $\Sy_n$ and $\psi(1)=\rho(\Al_n)$.
Let $\chi$ be an extension of $\psi$ to $\Sy_n$ and let $\lambda$
and $Y$ be respectively the partition and the Young diagram
associated to $\chi$. By the branching rule, we have
\[
  \chi^{\Sy_{n+1}}=\sum_{(i,j)\in A(\lambda)} \chi_{Y
  \cup\{(i,j)\}}.
\]

Assume that all the Young diagrams in $\{Y\cup \{(i,j)\}\mid
(i,j)\in A(\lambda)\}$ are non-symmetric. Then all the irreducible
characters $\chi_{Y \cup \{(i,j)\}}$ where $(i,j)\in A(\lambda)$
restrict irreducibly to $\Al_{n+1}$, and thus \[\chi_{Y \cup
\{(i,j)\}}(1)\leq \rho(\Al_{n+1}).\] We therefore deduce that
\[\chi^{\Sy_{n+1}}(1)\leq |A(\lambda)|\rho(\Al_{n+1}).\]
Since $|A(\lambda)|< \sqrt{2n}+1$ and
$\chi^{\Sy_{n+1}}(1)=(n+1)\rho(\Al_n)$, it follows that
$(n+1)\rho(\Al_n)< (\sqrt{2n}+1)\rho(\Al_{n+1})$, and hence
\[\rho(\Al_{n+1})> \frac{n+1}{ \sqrt{2n}+1}\rho(\Al_n).\]
When $n\geq 75$, we can check that
$(n+1)/(\sqrt{2n}+1)>(n+1)^{3/8}$. Therefore we conclude that
$\rho(\Al_{n+1})>(n+1)^{3/8}\rho(\Al_n)$, as desired.

It remains to assume that there is a symmetric Young diagram of the
form $Y\cup \{(i,j)$ with $(i,j)\in A(\lambda)$. Then there is
exactly one such diagram and at most $\sqrt{2n}$ non-symmetric
diagrams in $\{Y\cup \{(i,j)\}\mid (i,j)\in A(\lambda)\}$. Let $Y'$
be that symmetric Young diagram and $\mu$ be the corresponding
partition. By the branching rule, we have
\[{\chi_{Y'}}\downarrow_{\Sy_n}=\sum_{(i,j)\in R(\mu)} \chi_{Y'\backslash \{ (i,j) \}}.\]
We distinguish two cases:

\medskip

(1) All the Young diagrams of the form $Y'\backslash \{ (i,j) \}$
where $(i,j)\in R(\mu)$ are non-symmetric. Then the characters
associated to these diagrams restrict irreducibly to $\Al_n$ and
thus their degrees are at most $\rho(\Al_n)$. As $|R(\mu)|
<\sqrt{2n+2}$, we deduce that
\[\chi_{Y'}(1)= \sum_{(i,j)\in R(\mu)} \chi_{Y'\backslash \{ (i,j)
  \}}(1)< \sqrt{2n+2} \rho(\Al_n).\]
We then have
\begin{align*} (n+1)\rho(\Al_n)&= \sum_{(i,j)\in A(\lambda)} \chi_{Y \cup
\{(i,j)\}}(1)\\
&=\chi_{Y'}(1)+ \sum_{(i,j)\in A(\lambda), Y \cup \{(i,j)\}\neq Y' }
\chi_{Y \cup \{(i,j)\}}(1)\\
&< \sqrt{2n+2} \rho(\Al_n)+\sum_{(i,j)\in A(\lambda),~Y \cup
\{(i,j)\}\neq Y' } \chi_{Y \cup \{(i,j)\}}(1).
\end{align*}
Since $\chi_{Y \cup \{(i,j)\}}(1)\leq \rho(\Al_{n+1})$ whenever $Y
\cup \{(i,j)\}\neq Y'$, it follows that
\begin{align*}
(n+1)\rho(\Al_n)&<\sqrt{2n+2}
\rho(\Al_n)+(|\Al(\lambda)|-1)\rho(\Al_{n+1})\\&< \sqrt{2n+2}
\rho(\Al_n)+\sqrt{2n} \rho(\Al_{n+1}).
\end{align*}
Thus
\[\rho(\Al_{n+1})> \frac{ n+1- \sqrt{2n+2}}{
  \sqrt{2n}} \rho(\Al_n).\] Again, as $n\geq 75$ we now can easily
  deduce that $\rho(\Al_{n+1})>(n+1)^{3/8}\rho(\Al_n)$.

\medskip

(2) There is a symmetric Young diagram of the form $Y'\backslash
\{(i,j)\}$ where $(i,j)\in R(\mu)$. Let $Y''$ be this symmetric
Young diagram and $\nu$ be the associated partition. Then $Y''$ is
the only one symmetric diagram and there are at most $\sqrt{2n+2}-1$
non-symmetric diagrams in $\{Y'\backslash \{(i,j)\}\mid (i,j)\in
R(\mu)\}$. So we have two symmetric Young diagrams $Y'$ and $Y''$
and $Y''$ is obtained from $Y'$ by removing a node. Therefore, if
another node is removed from $Y''$ to get a Young diagram (of size
$n-1$), the resulting diagram cannot be symmetric. Therefore, by the
branching rule,
\[\chi_{Y''}(1)< \sqrt{2n} \rho(\Al_{n-1}).\]
It follows that
\[\chi_{Y'}(1)< \sqrt{2n} \rho(\Al_{n-1})+ (\sqrt{2n+2}-1) \rho(\Al_n).\]
Therefore,
\begin{align*} (n+1)\rho(\Al_n)
&=\chi_{Y'}(1)+ \sum_{(i,j)\in A(\lambda), Y \cup \{(i,j)\}\neq Y' }
\chi_{Y \cup \{(i,j)\}}(1)\\
&<  \sqrt{2n} \rho(\Al_{n-1})+ (\sqrt{2n+2}-1) \rho(\Al_n)+
\sqrt{2n} \rho(\Al_{n+1}).
\end{align*}
Using the induction hypothesis that $\rho(\Al_{n-1})\leq
n^{-3/8}\rho(\Al_n)$, we then have
\[\rho(\Al_{n+1})> \frac{ n+2- \sqrt{2n+2}- \sqrt{2n} n^{-3/8}}{\sqrt{2n}} \rho(\Al_n).\] Now with $n\geq 75$ we can check that
\[\frac{ n+2- \sqrt{2n+2}- \sqrt{2n} n^{-3/8}}{\sqrt{2n}}>(n+1)^{3/8},\] and the proof is complete.
\end{proof}


\section{Simple groups}\label{section:simple groups}

In this section we prove
Theorem~\ref{theorem-nonabelian-simple-groups}, and then deduce
Theorems~\ref{theorem-main} and \ref{theorem-main-3} for
characteristically simple groups. This will be used in the proof for
arbitrary groups. We restate
Theorem~\ref{theorem-nonabelian-simple-groups} here.

\begin{theorem}\label{theorem-nonabelian-simple-groups-1}
Let $S$ be a non-abelian simple group. Then $|S|>2b(S)^2$.
Consequently, if $|S| = d(d+e)$ where $d$ is the degree of some
irreducible character of $S$ then $|S| < 2e^2$.
\end{theorem}

Let $\Irr(G)$ denote the set of irreducible character of $G$.
Motivated by the problem of improving Snyder's bound,
Isaacs~\cite{Isaacs2} introduced and studied the invariant
\[\varepsilon(S) := \frac{ \sum_{ \chi \in \Irr (S),\, \chi(1) < b(S)} \chi(1)^2}{b(S)^2} \] for non-abelian simple groups $S$. He raised the question whether the largest
character degree of $S$ can be bounded in terms of smaller degrees
in the sense that $\varepsilon (S) \geq \varepsilon$ for some
universal constant $\varepsilon > 0$ and for all non-abelian simple
groups $S$. This was answered in the affirmative
in~\cite{Larsen-Malle-Tiep} with the bounding constant $\varepsilon$
taken to be $2/(120\,000!)$. We note that this rather small bound
comes from the alternating groups, see~\cite[Theorem~2.1 and
Corollary~2.2]{Larsen-Malle-Tiep} for more details.

To further improve the bound from $Be^6$ to $e^6 + e^4$, Isaacs even
predicted that $\varepsilon(S) > 1$ for every non-abelian simple
group $S$. This was in fact confirmed in~\cite{Larsen-Malle-Tiep}
for the majority of simple classical groups, and for all simple
exceptional groups of Lie type as well as sporadic simple groups.
Recently, Z.~Halasi, C.~Hannusch, and the first author have
confirmed that indeed $\varepsilon (\Al_n)> 1$ for every $n \geq 5$,
see \cite{Halasi-Hannusch-Nguyen}.

One easily sees that if $\varepsilon(S)> 1$ then $e>b(S)\geq d$ so
that $2b(S)^2<|S| < 2e^2$, and
Theorem~\ref{theorem-nonabelian-simple-groups-1} is proved for the
simple group $S$. Furthermore, when $S$ has a unique irreducible
character of the largest degree $b(S)$, $|S|>2b(S)^2$ is equivalent
to $\varepsilon(S) > 1$. Therefore
Theorem~\ref{theorem-nonabelian-simple-groups-1} can be viewed as a
weak form of Isaacs's prediction.

To prove Theorem~\ref{theorem-nonabelian-simple-groups-1}, we will
use Lusztig's classification of complex irreducible characters of
finite groups of Lie type (see \cite[Chapter~13]{Digne-Michel} and
\cite[\S13.8]{Carter}) and detailed structure of the centralizers of
semisimple elements in finite classical groups (see for instance
\cite[Section~3]{Tiep-Zalesskii}, \cite[Section~2]{Ng}, and
\cite[Section~2]{Burkett-Nguyen}). We first record two observations.

\begin{lemma}\label{lem:productbd}
Let $q\geq 2$.  Then $\displaystyle{\prod_{i=2}^\infty (1 - 1/q^i) >
9/16}$.
\end{lemma}

\begin{proof}
This is \cite[Lemma 4.1(ii)]{Larsen-Malle-Tiep}.
\end{proof}

Let $f\in \FF_q[t]$ be an irreducible monic polynomial.  In what
follows, we will write $\widetilde{f}$ for the polynomial over
$\FF_q[t]$ whose roots are $\{\alpha^{-1}|\alpha \hbox{ is a root of
$f$}\}$. Note that if $f=\widetilde{f}$, then the $\deg(f)$ is
necessarily even.  Moreover, \cite[Theorem~3]{Meyn-Gotz} gives a
formula for the number of $f$ satisfying $f=\widetilde{f}$, which
yields the following lemma.

\begin{lemma}\label{lem:srim}
Let $S_2(d)$ be the number of irreducible monic polynomial over
$\FF_2$ of degree $2d$ satisfying $f=\widetilde{f}$.  Then
$S_2(1)=S_2(2)=S_2(3)=1$; $S_2(4)=2$; $S_2(5)=3$;  $S_2(6)=5$;
$S_2(7)=9$; and $S_2(d)\geq 16$ for $d\geq 8$.
\end{lemma}

\begin{proof}
This is straightforward from \cite[Theorem~3]{Meyn-Gotz}.
\end{proof}

\begin{proof}[Proof of Theorem~\ref{theorem-nonabelian-simple-groups-1}]
Since the inequality $\varepsilon(S)> 1$ has been established for
all the simple exceptional groups of Lie type, the sporadic simple
groups, and the alternating groups, it remains to prove the theorem
for the simple classical groups.

Further, we note that we only need to consider those classical
groups of Lie type excluded from
\cite[Theorem~4.7]{Larsen-Malle-Tiep}.   That is, we must consider
the simple groups found in the following list:
\[\SL_n(2), \Sp_{2n}(2), \Omega_{2n}^\pm(2),\]
\[\PSL_n(3) \hbox{ with $5\leq n\leq 14$}, \PSU_n(2) \hbox{ with $7\leq
n\leq 14$}, \]\[\PSp_{2n}(3)\hbox{ or } \Omega_{2n+1}(3) \hbox{with
$4\leq n\leq 17$}, \PO^{\pm}_{2n}(3) \hbox{ with $4\leq n\leq30$},
\]\[\PO^\pm_8(7), \hbox{ and } \PO_{2n}^\pm(5)\hbox{with
$4\leq n\leq 6$}.\] We will make use of some of the ideas used in
\cite{Larsen-Malle-Tiep}, as well as the list of character degrees
of small rank groups of Lie type available on F.~L\"{u}beck's
website \cite{Lubeck}.

When the rank is at most $8$ all the character degrees of the simply connected group $G$ of the same
type as $S$ in
this list can be found from \cite{Lubeck}, and one can use this to check that in fact
$|S| >2b(G)^2\geq 2b(S)^2$, which implies that $e> b(S)$ and hence $|S|< 2e^2$.
So we assume that $S$ is one of the groups listed above with $n\geq
9$ (and $n\geq 10$ for type $A$).

\medskip

(1)  First, let $S$ be $\PSL_n(3)$, $\PSU_n(2)$, $\PSp_{2n}(3)$,
$\Omega_{2n+1}(3)$, $\PO^{\pm}_{2n}(3)$, $\PO^\pm_8(7)$, or
$\PO_{2n}^\pm(5)$, with $n$ as above, but larger than $8$. Note that
by \cite[Theorem~2.1]{Seitz}, \[b(S)\leq b(G)\leq |G:T|_{q'},\]
where $q$ is the size of the underlying field for $S$, $G$ is the
group of fixed points for the simple simply connected algebraic
group corresponding to $S$, and $T$ is a maximal torus of $G$ of
minimal order.  The size of $T$ is $(q-1)^n$ (or $(q-1)^{n-1}$ for
$\PSL_n(q)$) if $S$ is of untwisted type, and  can be found, for
example, in \cite[Table~1]{Larsen-Malle-Tiep} if $S$ is of twisted
type.  We may check directly using this bound for $b(S)$ that in
fact, $|S|>2b(S)^2$ for each group in this finite list.  This shows
that if $S$ is one of \[ \PSL_n(3) \hbox{ with $5\leq n\leq 14$},
\PSU_n(2) \hbox{ with $7\leq n\leq 14$},
\]\[\PSp_{2n}(3)\hbox{ or } \Omega_{2n+1}(3) \hbox{with $4\leq n\leq
17$}, \PO^{\pm}_{2n}(3) \hbox{ with $4\leq n\leq30$},
\]\[\PO^\pm_8(7), \hbox{ and } \PO_{2n}^\pm(5)\hbox{with
$4\leq n\leq 6$},\] then $2b(S)^2<|S|<2e^2$.

\medskip

(2) Now let $S$ be one of the groups $\SL_n(2)$, $\Sp_{2n}(2),$ or
$\Omega_{2n}^\pm(2)$, and assume $n\geq 10$ in the first case and
$n\geq 9$ in the latter two cases.  Then $S^\ast\cong S$ is
self-dual and the center of the corresponding algebraic group is
trivial.  We make the identification $S^\ast\cong S$, and hence by Lusztig's classification of complex irreducible
characters of finite groups of Lie type, $\mathrm{Irr}(S)$ is
parametrized by pairs $((s), \theta)$, where $(s)$ is a semisimple
conjugacy class in $S$ and
$\theta\in\mathrm{Irr}(\bC_{S}(s))$ is a unipotent character.
Further, the character parametrized by $((s), \theta)$ has degree
\[[S: \bC_{S}(s)]_{2'}\theta(1).\]

Notice that if there are at least two $\chi\in\mathrm{Irr}(S)$
satisfying $\chi(1)=b(S)$, then certainly $|S|>2b(S)^2$, and hence
$|S|<2e^2$.  Therefore, we may further assume that there is a unique
such $\chi$.

Notice that the centralizer of a semisimple element $s$ of $S$
is of the form \[\bC_{S}(s)\cong K\times H_1\times....\times
H_r,\] where each $H_i$ is of the form
$\GL_{k_i}^{\epsilon_i}(2^{d_i})$, $\epsilon_i$ is $+$ in the linear
case and $\pm$ for the symplectic and orthogonal cases, $K$ is
trivial in the linear case, $\Sp_{2m}(2)$ in the symplectic case,
and in the orthogonal case, we may assume by the argument toward the
beginning of \cite[Part (3) of Proof of Theorem
4.8]{Larsen-Malle-Tiep} that $K$ is $\Omega_{2m}^\pm(2)$.  (Indeed, by
\cite[Theorem 3.7]{Tiep-Zalesskii}, $\bC_{S}(s)\cong K_1\times H_2\times...\times H_r$
where each $H_i$ is as described above, and $K_1$ has a normal subgroup isomorphic to $\Omega_{2m}^\pm(2)$
with either trivial quotient or quotient isomorphic to $\GU_2(2)$.  Since $\mathrm{St}_{K_1}$ in the latter case has degree $2^{m(m-1)+1}$,
there is no loss in assuming $\bC_{S}(s)\cong K\times H_1\times H_2\times...\times H_r$ with $K$ as stated.)
Note
that we use the notation $\GL_k^+(2^d):=\GL_k(2^d)$ and
$\GL_k^-(2^d):=\GU_k(2^d)$. Further, $\sum k_id_i+m=n$, and the $K$
and $H_i$ are determine by the elementary divisors of $s$ acting on
the natural module $\mathbb{F}_2^n$ or $\mathbb{F}_2^{2n}$ for
$S$.  Namely, if $S=\Sp_{2n}(2)$ or $\Omega_{2n}^\pm(2)$, a
factor of $H_i\cong \GL_{k_i}(2^{d_i})$ corresponds to a pair of
monic polynomials $g_i(t)\widetilde{g}_i(t)$ in $\mathbb{F}_2[t]$
with multiplicity $k_i$, where $g_i\neq \widetilde{g_i}$ are
irreducible of degree $d_i$.  Moreover, $H_i\cong
\GU_{k_i}(d^{d_i})$ corresponds to a monic irreducible $f_i(t)\neq
t-1$ with degree $2d_i$ and multiplicity $k_i$, where
$f=\widetilde{f}$.  In these cases, $K$ corresponds to the
elementary divisor $t-1$, with multiplicity $2m$. If $S=\SL_{n}(2)$,
each elementary divisor $f_i(t)$ with degree $d_i$ and multiplicity
$k_i$ yields a factor $H_i\cong \GL_{k_i}(2^{d_i})$.

Let $\chi\in\mathrm{Irr}(S)$ satisfying $\chi(1)=b(S)$ be
parametrized by $((s), \theta)$. Then by
\cite[Theorem~1.2]{Larsen-Malle-Tiep}, $\theta$ must be the
Steinberg character $\mathrm{St}_{\bC_{S}(s)}$ of
$\bC_{S}(s)$. Recall that the Steinberg character of
$\GL_k^\pm(2^d)$ has degree $2^{dn(n-1)/2}$, the Steinberg character
of $\Sp_{2m}(2)$ has degree $2^{m^2}$, and the Steinberg character
of $\Omega^\pm_{2m}(2)$ has degree $2^{m(m-1)}$.

Moreover, by our assumption that $\chi$ is the unique member of
$\mathrm{Irr}(S)$ satisfying $\chi(1)=b(S)$, we see that it must be
the case that every polynomial of a given degree and type as
described above must appear as an elementary divisor of $s$ with the
same multiplicity.  (Indeed, otherwise, we may find another
semisimple element $s'\in S$ not conjugate to $s$ with
$\bC_{S}(s)\cong \bC_{S}(s')$, and hence the character
parametrized by $((s'), \mathrm{St}_{\bC_{S}(s')})$ has degree
$b(S)$ as well.)

We will proceed using some estimates for the number of monic
irreducible polynomials of a given type as above.

Let $S=\Omega_{2n}^\epsilon(2)$.  We present the complete proof in
this case and note that the proof in the other two cases are
similar, though less complicated.

\medskip

(3) First, if no factors of the form $\GL_{k_i}^\pm(2^{d_i})$
appears
in $\bC_{G^\ast}(s)$, then we see that $\chi=\mathrm{St}$ has degree $2^{n(n-1)}$.
Otherwise, write $\bC_{S}(s)\cong \Omega_{2m}^\beta(2)\times \GL_{k_1}^{\epsilon_1}(2^{d_1}) \times \GL_{k_2}^{\epsilon_2}(2^{d_2})
\times\ldots\times \GL_{k_r}^{\epsilon_r}(2^{d_r})$ with $r\geq 1$. In this case, 

\begin{eqnarray*}
\chi(1)&=&2^{m(m-1)+\sum_{\ell=1}^r d_\ell k_\ell(k_\ell-1)/2}\frac{(2^n-\epsilon)\prod_{j=m}^{n-1}(2^{2j}-1)}{(2^m-\beta)\prod_{\ell=1}^r\left(\prod_{i=1}^{k_\ell}(2^{id_\ell}-\epsilon_\ell^{i})\right)}\\
&=& 2^{m(m-1)+\sum_{\ell=1}^r d_\ell k_\ell(k_\ell-1)/2}\frac{(2^m+\beta)\prod_{j=m+1}^{n}(2^{2j}-1)}{(2^n+\epsilon)\prod_{\ell=1}^r\left(\prod_{i=1}^{k_\ell}(2^{id_\ell}-\epsilon_\ell^{i})\right)}\\
&=& \frac{2^{m(m-1)+\sum_{\ell=1}^r d_\ell k_\ell(k_\ell-1)/2}}{2^{\sum_{\ell=1}^r d_\ell k_\ell(k_\ell+1)/2}}\frac{(2^m+\beta)\prod_{j=m+1}^{n}(2^{2j}-1)}{(2^n+\epsilon)\prod_{\ell=1}^r\left(\prod_{i=1}^{k_\ell}(1-(\epsilon_\ell/2^{d_\ell})^i)\right)}\\
&\leq &\left(\frac{16}{9}\right)^r\left(\frac{2^m+1}{2^n-1}\right) \frac{2^{n(n+1)-m(m+1)+m(m-1)+\sum_{\ell=1}^r d_\ell k_\ell(k_\ell-1)/2}}{2^{\sum_{\ell=1}^r d_\ell k_\ell(k_\ell+1)/2}}\\
&=& \left(\frac{16}{9}\right)^r\left(\frac{1+1/2^m}{1-1/2^n}\right) \frac{2^{m+n(n+1)-m(m+1)+m(m-1)+\sum_{\ell=1}^r d_\ell k_\ell(k_\ell-1)/2}}{2^{n+\sum_{\ell=1}^r d_\ell k_\ell(k_\ell+1)/2}}\\
&=&\left(\frac{16}{9}\right)^r\left(\frac{1+1/2^m}{1-1/2^n}\right)2^{n(n-1)}\\
&\leq& \left(\frac{16}{9}\right)^r\left(\frac{3}{2}\right)\left(\frac{512}{511}\right)2^{n(n-1)}.\\
\end{eqnarray*}
Note that the bound remains true if $m=0$, and that we have used
Lemma \ref{lem:productbd} and the fact that $n\geq 9$.  If $0\leq
r\leq 3$, this calculation (together with the first observation for
the case $r=0$) yields that
\[\chi(1)< 9\cdot 2^{n(n-1)},\] so we see
\[\frac{|S|}{\chi(1)^2}> \left(\frac{1}{9}\right)^2\frac{(2^n-1)\prod_{i=1}^{n-1}(2^{2i}-1)}{2^{n(n-1)}}=
\frac{1}{81}\frac{(2^n-1)2^{n(n-1)}\prod_{i=1}^{n-1}(1-1/2^{2i})}{2^{n(n-1)}}\]
and by Lemma~\ref{lem:productbd},
\[\frac{|S|}{\chi(1)^2}>  \left(\frac{1}{81}\right)\left(\frac{9}{16}\right)(2^n-1),\] which is larger than $2$ since $n\geq 9$.
Hence we see that $|S|>2\chi(1)^2=2b(S)^2$ if $r\leq 3$. We may
therefore assume that \[\bC_{S}(s)\cong
\Omega_{2m}^\beta(2)\times \GL_{k_1}^{\epsilon_1}(2^{d_1}) \times
\GL_{k_2}^{\epsilon_2}(2^{d_2}) \times\ldots\times
\GL_{k_r}^{\epsilon_r}(2^{d_r})\] with $r\geq 4$, and assume
$d_1k_1\geq d_2k_2\geq...\geq d_rk_r$.

\medskip

(4) Our strategy for the remainder of the proof is to consider
semisimple elements $t\in S$ and the characters $\psi$
corresponding to $(t, \mathrm{St}_{\bC_{S}(t)})$.  We will show
that there are a sufficient number of such semisimple elements with
$\psi(1)/\chi(1)$ large enough to imply that $\epsilon(S) >1$, and
therefore that $|S|<2e^2$.

Let $\mathfrak{F}$ denote the set of all monic polynomials $f \neq
t-1$ over $\mathbb{F}_2$ which are either irreducible satisfying
$f=\widetilde{f}$ or of the form $f=g\widetilde{g}$ where $g\neq
\widetilde{g}$ are irreducible.  For $f\in \mathfrak{F}$, write
$\epsilon_f=-1$ if $f$ is irreducible and $\epsilon_f=1$ if
$f=g\widetilde{g}$, and write $d_f$ for the degree of $f$.  Then
given that $\kappa\colon \mathfrak{F}\rightarrow \mathbb{N}$ is a
function satisfying $n-m=\frac{1}{2}\sum_{f\in\mathfrak{F}}
d_f\kappa(f)$ and $\prod_{f\in\mathfrak{F}}
(\epsilon_f)^{\kappa_f}=\prod_{i=1}^r (\epsilon_i)^k_i $, there
exists a semisimple $t\in S$ with corresponding multiplicities
$\kappa(f)$ for the polynomials $f$ as elementary divisors, and
hence \[\bC_{S}(t)\cong \Omega_{2m}^\beta(2)\times \prod_{f\in
\mathfrak{F}} \GL_{\kappa(f)}^{\epsilon_f}(2^{d_f/2}).\]

Now, notice that there are at least two pairs $(i,j)$ with $4\geq i>
j\geq 1$ such that $d_ik_i+d_jk_j$ is even.  Moreover, these pairs
satisfy $d_ik_i+d_jk_j\geq 4$ since the combination $(d_\ell,
k_\ell)=(1,1)$ can occur at most once.  Also note that if a factor
of the form $\GL_k^\pm(2^{k_id_i+k_jd_j})$ appears in
$\bC_{S}(s)$, it must be that $(k, k_id_i+k_jd_j)=(k_1, d_1)$
or $(k_2, d_2)$ and if $\GU_k(2^{(k_id_i+k_jd_j)/2})$ appears, then
$(k_id_i+k_jd_j)/2\in\{d_1,..,d_4\}$ (and correspondingly
$k\in\{k_1,...,k_4\}$).

We consider four situations:

\smallskip

(i)  $\GL^\epsilon_k(2^{k_id_i+k_jd_j})$ is not a factor of
$\bC_{S}(s)$, in which case we will consider a semisimple
element $t\in S$ with
\[\bC_{S}(t)\cong \Omega_{2m}^\beta(2)\times \GL^\epsilon_{1}(2^{k_id_i+k_jd_j}) \times \prod_{\ell\in\{1,...,r\}\setminus\{i,j\}} \GL_{k_\ell}^{\epsilon_\ell}(2^{d_\ell}).\]

\smallskip

(ii) $\GL^\epsilon_k(2^{k_id_i+k_jd_j})$ is a factor of
$\bC_{S}(s)$, in which case we will consider a semisimple
element $t\in S$ with
\[\bC_{S}(t)\cong \Omega_{2m}^\beta(2)\times \GL^\epsilon_{k+1}(2^{k_id_i+k_jd_j}) \times \GL_{k_\ell}^{\epsilon_\ell}(2^{d_\ell})
\times \GL_{k_5}^{\epsilon_5}(2^{d_5})\ldots\times
\GL_{k_r}^{\epsilon_r}(2^{d_r})\] where we write $\ell\in\{1,..,4\}$
so that $\ell\neq i, j$, or the index corresponding to $(k,
k_id_i+k_jd_j)$.

\smallskip

(iii) $\GU_k(2^{(k_id_i+k_jd_j)/2})$ is not a factor of
$\bC_{S}(s)$, in which case we consider a semisimple element
$t\in S$ with
\[\bC_{S}(t)\cong \Omega_{2m}^\beta(2)\times \GU_{2}(2^{(k_id_i+k_jd_j)/2}) \times \prod_{\ell\in\{1,...,r\}\setminus\{i,j\}} \GL_{k_\ell}^{\epsilon_\ell}(2^{d_\ell}).\]

\smallskip

(iv) $\GU_k(2^{(k_id_i+k_jd_j)/2})$ is a factor of
$\bC_{S}(s)$, in which case we consider a semisimple element
$t\in S$ with
\[\bC_{S}(t)\cong \Omega_{2m}^\beta(2)\times \GU_{k+2}(2^{(k_id_i+k_jd_j)/2}) \times \GL_{k_\ell}^{\epsilon_\ell}(2^{d_\ell}) \times \GL_{k_5}^{\epsilon_5}(2^{d_5})\ldots\times \GL_{k_r}^{\epsilon_r}(2^{d_r})\]
where we write $\ell\in\{1,..,4\}$ so that $\ell\neq i, j$, or the
index corresponding to $(k, (k_id_i+k_jd_j)/2)$.

Note that for situations (iii) and (iv), it must be that
$\epsilon_i^{k_i}\epsilon_j^{k_j}=1$.  In each situation, we will
let $\psi\in\mathrm{Irr}(S)$ correspond to $(t,
\mathrm{St}_{\bC_{S}(t)})$, and arrive at lower bounds for
$\frac{\psi(1)}{\chi(1)}$.  Note that from the last paragraph of
part (3) of the proof of \cite[Theorem~4.8]{Larsen-Malle-Tiep}, we
have that in situation (i), $\psi(1)/\chi(1)>\frac{81}{320}$.  We
use similar arguments in the remaining situations.

Consider situation (ii).  For simplicity in the calculation, rewrite
$(i,j)$ as $(1,2)$ and write $d_0:=d_1k_1+d_2k_2$.  Then

\begin{eqnarray*}
\frac{\psi(1)}{\chi(1)}&=&
\frac{2^{d_0k(k+1)/2}\prod_{\nu=1}^{k_1}(2^{\nu
d_1}-(\epsilon_1)^{\nu})
\prod_{\nu=1}^{k_2}(2^{\nu d_2}-(\epsilon_2)^{\nu})\prod_{\nu=1}^k(2^{\nu d_0}-\epsilon^\nu)}{2^{d_0k(k-1)/2+d_1k_1(k_1-1)/2+d_2k_2(k_2-1)/2} \prod_{\nu=1}^{k+1}(2^{\nu d_0}- \epsilon^\nu)}  \\
&=&  \frac{2^{d_0k}\prod_{\nu=1}^{k_1}(2^{\nu d_1}-(\epsilon_1)^{\nu})\prod_{\nu=1}^{k_2}(2^{\nu d_2}-(\epsilon_2)^{\nu}) }{2^{d_1k_1(k_1-1)/2+d_2k_2(k_2-1)/2}(2^{k d_0+d_0}-\epsilon^{k+1})}  \\
&> & \frac{4}{5}\cdot \frac{\prod_{\nu=1}^{k_1}(2^{\nu d_1}-(\epsilon_1)^{\nu})\prod_{\nu=1}^{k_2}(2^{\nu d_2}-(\epsilon_2)^{\nu}) }{2^{d_1k_1(k_1-1)/2+d_2k_2(k_2-1)/2+d_0}}  \\
&=&\frac{4}{5}\cdot \frac{2^{d_1k_1(k_1+1)/2+d_2k_2(k_2+1)/2}\prod_{\nu=1}^{k_1}(1-(\epsilon_1/2^{ d_1})^\nu)\prod_{\nu=1}^{k_2}(1-(\epsilon_2/2^{d_2})^\nu) }{2^{d_1k_1(k_1-1)/2+d_2k_2(k_2-1)/2+d_0}}  \\
&=& \frac{4}{5}\cdot2^{d_1k_1+d_2k_2-d} \cdot \prod_{\nu=1}^{k_1}(1-(\epsilon_1/2^{ d_1})^\nu)\prod_{\nu=1}^{k_2}(1-(\epsilon_2/2^{d_2})^\nu)\\
&> & \frac{4}{5}\cdot (9/16)^2=\frac{81}{320}\\
\end{eqnarray*}
by Lemma~\ref{lem:productbd}, since $(d_j, \epsilon_j)\neq (1,1)$
for any $j$. In the third line, we have also used the fact that
$2^{kd+d}+1\leq \frac{5}{4} 2^{kd_0+d_0}$ since certainly
$kd+d\geq2$.

Now, consider situation (iii), and again for simplicity rewrite
$(i,j)$ as $(1,2)$ and write $d_0:=d_1k_1+d_2k_2$.  Then

\begin{eqnarray*}
\frac{\psi(1)}{\chi(1)} &=&\frac{2^{d_0/2}\cdot \prod_{\nu=1}^{k_1}(2^{\nu d_1}-(\epsilon_1)^{\nu})\prod_{\nu=1}^{k_2}(2^{\nu d_2}-(\epsilon_2)^{\nu})}{2^{d_1k_1(k_1-1)/2+d_2k_2(k_2-1)/2}\cdot (2^{d_0/2}+1)(2^{d_0}-1)}  \\
&>& \frac{2^{d_0/2}}{2^{d_0/2}+1}\cdot \frac{\prod_{\nu=1}^{k_1}(2^{\nu d_1}-(\epsilon_1)^{\nu})\prod_{\nu=1}^{k_2}(2^{\nu d_2}-(\epsilon_2)^{\nu}) }{2^{d_1k_1(k_1-1)/2+d_2k_2(k_2-1)/2+d_0}}  \\
&>& \frac{81}{256}\cdot \frac{2^{d_0/2}}{2^{d_0/2}+1}\\
&\geq & \frac{81}{256} \cdot \frac{4}{5} = \frac{81}{320}\\
\end{eqnarray*}
where the last inequality is since $d_0\geq 4$, and the
second-to-last is by the same argument as situation (ii).

Finally, consider situation (iv).  As before, write $(i,j)$ as
$(1,2)$, and $d_0:=d_1k_1+d_2k_2$.  We have

\begin{eqnarray*}
\frac{\psi(1)}{\chi(1)} &=&\frac{2^{d_0(k+2)(k+1)/4}\cdot
\prod_{\nu=1}^{k_1}(2^{\nu
d_1}-(\epsilon_1)^{\nu})\prod_{\nu=1}^{k_2}(2^{\nu
d_2}-(\epsilon_2)^{\nu})
\prod_{\nu=1}^k(2^{\nu d/2}-(-1)^\nu)}{2^{dk(k-1)/4+d_1k_1(k_1-1)/2+d_2k_2(k_2-1)/2}\cdot \prod_{\nu=1}^{k+2}(2^{\nu d_0/2}-(-1)^\nu)}  \\
&=&\frac{2^{d_0(k+2)(k+1)/4}\cdot \prod_{\nu=1}^{k_1}(2^{\nu
d_1}-(\epsilon_1)^{\nu})
\prod_{\nu=1}^{k_2}(2^{\nu d_2}-(\epsilon_2)^{\nu})}{2^{d_0k(k-1)/4+d_1k_1(k_1-1)/2+d_2k_2(k_2-1)/2}\cdot (2^{d_0(k+1)/2}-(-1)^{k+1})(2^{d_0(k+2)/2}-(-1)^{k+2})}  \\
\end{eqnarray*}
Now, notice that one of $k+1$ and $k+2$ is even, so that
\[(2^{d_0(k+1)/2}-(-1)^{k+1})(2^{d_0(k+2)/2}-(-1)^{k+2})\leq
\frac{17}{16} 2^{d_0(k+1)/2+d_0(k+2)/2}\] since $d_0(k+1)/2$ and
$d_0(k+2)/2$ are at least $4$.  Hence
\begin{eqnarray*}
\frac{\psi(1)}{\chi(1)}&\geq& \left(\frac{16}{17}\right)
\frac{2^{d_0(k+2)(k+1)/4}\cdot
\prod_{\nu=1}^{k_1}(2^{\nu d_1}-(\epsilon_1)^{\nu})\prod_{\nu=1}^{k_2}(2^{\nu d_2}-(\epsilon_2)^{\nu})}{2^{d_0k(k-1)/4+d_1k_1(k_1-1)/2+d_2k_2(k_2-1)/2}\cdot   2^{d_0(k+1)/2+d_0(k+2)/2}}\\
&=&\left(\frac{16}{17}\right)  \cdot \frac{\prod_{\nu=1}^{k_1}(2^{\nu d_1}-(\epsilon_1)^{\nu})\prod_{\nu=1}^{k_2}(2^{\nu d_2}-(\epsilon_2)^{\nu}) }{2^{d_1k_1(k_1-1)/2+d_2k_2(k_2-1)/2+d_0}}  \\
&>& \left(\frac{16}{17}\right) \left(\frac{81}{256}\right)=
\frac{81}{272}.
\end{eqnarray*}

Hence, in each situation, we see that $\psi(1)/\chi(1)\geq 81/320$.
Now, let $d_0:=d_ik_i+d_jk_j$ as above.  Suppose that
$\epsilon_i^{k_i}\epsilon_j^{k_j}=-1$.    Note that for every $f\in
\mathfrak{F}$ which is irreducible of degree $2d_0$, we can identify
a semisimple element $t$ as in situation (i) or (ii) with
$\epsilon=-1$.   By Lemma \ref{lem:srim}, there are at least 16 such
$f$ as long as $d_0\geq 8$, yielding at least 16 characters $\psi\in
\Irr(S)$ satisfying $\psi(1)/\chi(1)\geq 81/320$ when $d_0\geq 8$.
Hence if $d_0\geq 8$, we see that $\epsilon(S)\geq 16 (81/320)^2>1$,
and $|S|<2e^2$.

Now suppose $\epsilon_i^{k_i}\epsilon_j^{k_j}=-1$.  Define
$\mathfrak{F}_{d_0}\subset \mathfrak{F}$ to be the set of monic
polynomials of the form $g\widetilde{g}$ where  $g\neq
\widetilde{g}$ are irreducible of degree $d_0$ together with the
monic irreducible polynomials $f\neq t-1$ of degree $d_0$ such that
$f=\widetilde{f}$.   Notice that if $n_{d_0}$ is the number of
irreducible monic polynomials over $\FF_2$ of degree $d_0$, then
$|\mathfrak{F}_{d_0}|\geq n_{d_0}/2$.  Moreover, for each choice of
$\mathfrak{f}\in\mathfrak{F}_{d_0}$, we can identify a semisimple
element $t\in S$ as in one of the cases (i)-(iv), with
$\epsilon=1$ in cases (i) and (ii).  This yields at least
$n_{d_0}/2$ characters $\psi$ satisfying $\psi(1)/\chi(1)\geq
81/320$.  Note that by \cite[(5.1)]{Larsen-Malle-Tiep}, if $d_0\geq
3$, then $n_{d_0}\geq \frac{3\cdot2^{d_0}}{4d_0}$.     Then
certainly $|\mathfrak{F}_{d_0}|\geq \frac{n_{d_0}}{2}\geq
\frac{3\cdot2^{d_0}}{8d_0}$ as long as $d_0\geq 3$, which is at
least $12$ if $d_0\geq 8$.

So, if $d_0\geq 8$ for both choices of $(i,j)$ (recall there must be
at least two pairs $(i,j)$ with $d_0=d_ik_i+d_jk_j$ even), then
there are at least $24$ characters $\psi$ satisfying
$\psi(1)/\chi(1)\geq 81/320$, so that $\epsilon(S)\geq 24 \cdot
(81/320)^2>1$, and we see in this case that $|S|<2e^2$.

Finally, considering each possibility for
$\GL_{k_i}^{\epsilon_i}(2^{d_i})\times
\GL_{k_j}^{\epsilon_j}(2^{d_j})$ satisfying $d_ik_i+d_jk_j=4$ or
$6$, we can use similar (but now more explicit) calculations to show
that in each case, $\epsilon(S)>1$, completing the proof for $\Omega_{2n}^\pm(2)$.

We make a final remark about the proofs for $\Sp_{2n}(2)$ and
$\SL_n(2)$. In either case, calculations analogous to those in part
(3) above yield similar results.  The remainder of the proof for
$\Sp_{2n}(2)$ follows directly from the calculations in part (4)
above for $\Omega_{2n}^\pm(2)$, replacing $\Omega_{2m}^\beta(2)$
with $\Sp_{2m}(2)$.  The analogue to part (4) for $\SL_{n}(2)$ is
similar, but requires only considering case (i) above, with
$\epsilon=1$, together with the estimate for $n_d$, since each
elementary divisor of $s$ yields a factor $\GL_{k_i}(2^{d_i})$ in
this case.
\end{proof}

The next observation is useful in the proofs of the main results.

\begin{lemma}\label{lemma-N-G/N}
Let $N$ be a nontrivial proper normal subgroup of $G$. Assume that
$b (G) \leq b (N)b (G/N)$. Then Theorem~\ref{theorem-main} is true
for $G$. Furthermore, if $|G|=b(G)(b(G)+e)$ then $e>2\sqrt{b(G)}$.
\end{lemma}

\begin{proof}
Write $|N| = b(N)(b(N) + e(N))$, $|G/N| = b(G/N)(b(G/N) + e(G/N))$,
and recall that $|G| = b(G)(b(G) + e)$. Then
\[b(G)(b(G) + e) = b(N) b(G/N) (b(N) + e(N))(b(G/N) + e(G/N)).\]
As $b(G)\leq b(N)b(G/N)$, we deduce that \begin{align*}e&\geq
e(N)e(G/N)+e(N)b(G/N)+b(N)e(G/N)\\
&>e(N)b(G/N)+b(N)e(G/N)\\
&\geq 2\sqrt{b(N)b(G/N)}\\
&\geq 2\sqrt{b(G)}.\end{align*} Note that, as both $N$ and $G/N$ are
nontrivial, $e(N)>0$ and $e(G/N)>0$. We now easily deduce that
$|G|<e^4-e^3$.
\end{proof}

\begin{corollary}\label{corollary-nonabelian-simple-groups}
Theorems~\ref{theorem-main} and \ref{theorem-main-3} are true for
every finite group which is direct product of non-abelian simple
groups. In particular, they are true for all characteristically
simple groups.
\end{corollary}

\begin{proof}
This follows from Theorem~\ref{theorem-nonabelian-simple-groups-1}
and Lemma~\ref{lemma-N-G/N}.
\end{proof}


\section{The case $S\ncong \PSL_2(q)$}\label{section:S not PSL}

With Theorem~\ref{theorem3/8} in hand, we are now ready to prove the
main results in the case $S\ncong \PSL_2(q)$. First, we recall the
following lemma, which will be frequently used from now on.

\begin{lemma}\label{lemma-extension}
Let $N=S\times\cdots\times S$, a direct product of copies of a
non-abelian simple group $S$, be a minimal normal subgroup of $G$.
Assume that $\theta\in\Irr(S)$ is extendible to $\Aut(S)$. Then the
product character $\psi:=\theta\times\cdots\times\theta\in\Irr(N)$
is extendible to $G$. Consequently, if $\chi\in\Irr(G)$ is an
extension of $\psi$, then there is a bijection $\beta\leftrightarrow
\beta\chi$ between $\Irr(G/N)$ and the set of irreducible characters
of $G$ lying above $\psi$.
\end{lemma}

\begin{proof}
The first statement of the lemma is well known (see for instance
\cite[Lemma~5]{Bianchi-Lewis} or \cite[Lemma~1]{Moreto-Nguyen}). The
second statement follows by Gallagher's theorem, see
\cite[Corollary~6.17]{Isaacs1}.
\end{proof}

\begin{theorem}\label{theorem-S-ncongPSL(2,q)}
Let $G$ be a finite group with a minimal normal subgroup
$N=S\times\cdots\times S$, where $S$ is a non-abelian simple group
different from $\PSL_2(q)$ for every prime power $q$. Let
$|G|=b(G)(b(G)+e)$. Then $e>\sqrt{b(G)}+1$ and, in particular, $|G|<
e^4-e^3$.
\end{theorem}

\begin{proof}
Let $\theta$ be a character of $S$ found in
Theorem~\ref{theorem3/8}, i.e. $\theta$ is extendible to $\Aut(S)$
and $\theta(1)>|S|^{3/8}$. Let $\psi:=\theta\times\dots \times
\theta\in\Irr(N)$. Using Lemma~\ref{lemma-extension}, we deduce that
$\psi$ is extended to a character $\chi\in \Irr(G)$ and the mapping
$\beta\mapsto \beta\chi$ is a bijection between $\Irr(G/N)$ and the
set of irreducible characters of $G$ lying above $\psi\in \Irr(N)$.
This implies in particular that $\chi(1)b(G/N)$ is a character
degree of $G$, and whence $b(G)\geq \chi(1)b(G/N)$.

If $b(G)=\chi(1)b(G/N)$, then $b(G)\leq b(N)b(G/N)$ and we are done
by Lemma~\ref{lemma-N-G/N}. So for the rest of the proof we assume
that $b(G)>\chi(1)b(G/N)$. This means that the degree of any
irreducible character of $G$ lying above $\psi$ is less than $b(G)$.
We therefore deduce that
\[b(G)e=|G|-b(G)^2\geq \sum_{\beta\in\Irr(G/N)} (\chi(1)\beta(1))^2=\chi(1)^2|G/N|.\]
Using the fact that $\chi(1)=\theta(1)^k>|S|^{3k/8}=|N|^{3/8}$, we
then obtain
\[b(G)e>|N|^{3/4}|G/N|=|G|/|N|^{1/4}.\]

As the case $G=N$ has been already handled in
Corollary~\ref{corollary-nonabelian-simple-groups}, we may assume
that $|G/N|\geq 2$. Also note that $|G|\geq 2|N|\geq 5\,040$. We now
easily see that $|G|/|N|^{1/4}> |G|^{3/4}+|G|^{1/2}$. This and the
above inequality imply that
\[b(G)e>|G|^{3/4}+|G|^{1/2}.\]
Since $b(G)\leq |G|^{1/2}$, it follows that
\[b(G)e>b(G)^{3/2}+b(G),\] or equivalently
\[e>b(G)^{1/2}+1.\]
This implies that $b(G)<e^2-e$, which in turn implies that
\[|G|=b(G)(b(G)+e)<(e^2-e)e^2=e^4-e^3,\] and the theorem is
completely proved.
\end{proof}


\section{The case $S\cong \PSL_2(q)$ with $q$ even}
\label{section:S=PSL even}

Characters of the linear groups in dimension $2$ are well known and
we will use \cite{White} as the main source. In particular, we will
follow the notation there.

According to \cite[p.~8]{White}, when $q$ is even,
$\SL_2(q)\cong\PSL_2(q)$ has the following irreducible characters
\begin{enumerate}
\item[(i)] $1_{\SL_2(q)}$ of degree 1,
\item[(ii)] $\St_{\SL_2(q)}$ of degree $q$,
\item[(iii)] $\chi_i, 1\leq i\leq (q-2)/2$, of degree $q+1$, and
\item[(iv)] $\theta_j, 1\leq j\leq q/2$, of degree $q-1$.
\end{enumerate}

Let $q=2^f$ and $\varphi$ the field automorphism of order $f$ of
$\SL_2(q)$. Then, by \cite[Lemma~4.8]{White}, the character
$\chi_i\in\Irr(\SL_2(q))$ is invariant under $\varphi^k$ where
$1\leq k\leq f$ if and only if $(2^f-1)\mid i(2^k-1)$ or
$(2^f-1)\mid i(2^k+1)$; and the character
$\theta_j\in\Irr(\SL_2(q))$ is invariant under $\varphi^k$ if and
only if $(2^f+1)\mid j(2^k-1)$ or $(2^f+1)\mid j(2^k+1)$. Using
this, we can deduce that $\SL_2(2^f)$ has a non-principal
irreducible character besides the Steinberg character that is
extendible to $\Aut(\SL_2(2^f))$.

\begin{lemma}\label{lemma-theta-q-1-q+1}
The simple groups $\SL_2(q)$ with $q\geq 8$ even always have an
irreducible character $\theta$ of degree $q-1$ or $q+1$ such that
$\theta$ is extendible to $\Aut(\SL_2(q))$.
\end{lemma}

\begin{proof}
Assume that $q=2^f$ with $f\geq 3$. From the above discussion, we
observe that when $f$ is odd then $3\mid (2^f+1)$ and
$\theta_{(2^f+1)/3}$ is invariant under $\varphi$. On the other
hand, when $f$ is even then $3\mid (2^f-1)$ and $\chi_{(2^f-1)/3}$
is invariant under $\varphi$. So in any case, there is always an
irreducible character $\theta\in\Irr(\SL_2(q))$ of degree $q-1$ or
$q+1$ such that $\theta$ is invariant in $\Aut(\SL_2(q))$. Note that
$\Aut(\SL_2(2^f))=\SL_2(2^f)\rtimes \langle \varphi\rangle$. Thus
$\theta$ is extendible to $\Aut(\SL_2(q))$, as wanted.
\end{proof}

\begin{lemma}\label{lemma-b(G)}
Let $N=\PSL_2(q)\times\cdots\times \PSL_2(q)$, a direct product of
$k$ copies of the simple linear group $\PSL_2(q)$, is a normal
subgroup of $G$. Then
\[
b(G)\leq\min\{|G|^{1/2},(q+1)^k|G/N|\}.
\]
\end{lemma}

\begin{proof}
It is clear that $b(G)\leq |G|^{1/2}$, so it remains to show that
$b(G)\leq (q+1)^k|G/N|$. But this is also clear since
$b(\PSL_2(q))\leq q+1$ for every prime power $q\geq5$.
\end{proof}

We are now ready to prove Theorems~\ref{theorem-main} and
\ref{theorem-main-3} in the case $S\cong \PSL_2(q)$ with $q$ even.

\begin{theorem}\label{theorem-S-congPSL(2,q)q-even}
Assume that $N=\PSL_2(q)\times\cdots\times \PSL_2(q)$, a direct
product of $k$ copies of $\PSL_2(q)$ where $q\geq8$ is even, is a
minimal normal subgroup of a finite group $G$. Let
$|G|=b(G)(b(G)+e)$. Then $e>\sqrt{b(G)}+1$ and, in particular, $|G|<
e^4-e^3$.
\end{theorem}

\begin{proof}
Let $\theta\in\Irr(\SL_2(q))$ be an irreducible character of degree
$q-1$ or $q+1$ such that $\theta$ is extendible to $\Aut(\SL_2(q))$,
as its existence is guaranteed by Lemma~\ref{lemma-theta-q-1-q+1}.
Using Lemma~\ref{lemma-extension}, we obtain a bijection
$\beta\leftrightarrow \beta\chi$ between $\Irr(G/N)$ and the set of
irreducible characters of $G$ lying above
$\theta\times\cdots\times\theta\in\Irr(N)$, where $\chi$ is an
extension of $\theta\times\cdots\times\theta$ to $G$.

Consider the case $b(G)=\chi(1)b(G/N)$. We then have $b(G)\leq
b(N)b(G/N)$ and as in the proof of
Theorem~\ref{theorem-S-ncongPSL(2,q)}, we are done by
Lemma~\ref{lemma-N-G/N}. So we can assume that $b(G)>\chi(1)b(G/N)$.
In other words, all the irreducible characters of $G$ lying above
$\theta\times\cdots\times\theta\in\Irr(N)$ have degree smaller than
$b(G)$.

Repeat the above arguments for the Steinberg character
$\St_{\SL_2(q)}$ in place of $\theta$, we also can assume that all
irreducible characters of $G$ lying above
$\St_{\SL_2(q)}\times\cdots\times\St_{\SL_2(q)}\in\Irr(N)$ have
degree smaller than $b(G)$. Note that these characters are of the
form $\beta\chi_1$ where $\beta\in\Irr(G/N)$ and $\chi_1$ is an
extension of
$\St_{\SL_2(q)}\times\cdots\times\St_{\SL_2(q)}\in\Irr(N)$ to $G$.

The conclusions of the last two paragraphs imply that
\begin{align*}
b(G)e=|G|-b(G)^2&> \sum_{\beta\in\Irr(G/N)}
(\beta(1)^2\chi(1)^2+\beta(1)^2\chi_1(1)^2)\\
&=(\chi(1)^2+\chi_1(1)^2)|G/N|\\
&\geq ((q-1)^{2k}+q^{2k})|G/N|.
\end{align*}
It is straightforward to check that \[((q-1)^{2k}+q^{2k})|G/N|\geq
|G|^{3/4}+|G|^{1/2}\] if $|G/N|\geq q^k$, and
\[((q-1)^{2k}+q^{2k})|G/N|\geq
(q+1)^{3k/2}|G/N|^{3/2}+(q+1)^k|G/N|\] if $|G/N|< q^k$. Therefore,
it follows from Lemma~\ref{lemma-b(G)} that
\[((q-1)^{2k}+q^{2k})|G/N|\geq b(G)^{3/2}+b(G).\] We finally deduce
that $b(G)e> b(G)^{3/2}+b(G)$, and the desired inequality follows.
\end{proof}


\section{The case $S\cong \PSL_2(q)$ with $q$ odd}
\label{section:S=PSLodd}

We now turn to the most complicated case, namely $S\cong\PSL_2(q)$
with odd $q$. This will be achieved in
Theorem~\ref{theorem-S-congPSL(2,q)1} and
Theorem~\ref{theorem-S-congPSL(2,q)2}.

\begin{theorem}\label{theorem-S-congPSL(2,q)1}
Assume that $N=\PSL_2(q)\times\cdots\times \PSL_2(q)$, a direct
product of $k$ copies of $\PSL_2(q)$ where $q\geq5$ is an odd prime
power, is a minimal normal subgroup of a finite group $G$ such that
$|G/N|\geq q^k$. Let $|G|=b(G)(b(G)+e)$. Then $e>\sqrt{b(G)}+1$ and,
in particular, $|G|< e^4-e^3$.
\end{theorem}

\begin{proof}
Write $N = S_1 \times \cdots \times S_k$ where $S_i\cong \PSL_2(q)$
for every $i=1,2,...,k$. As before, we apply
Lemma~\ref{lemma-extension} to have a bijective map $\beta\mapsto
\beta\chi$ from $\Irr(G/N)$ to the set of irreducible characters of
$G$ lying above $\St_{S_1}\times\St_{S_2}\times\cdots\times
\St_{S_k}\in \Irr(N)$, where $\chi$ is an extension of
$\St_{S_1}\times\St_{S_2}\times\cdots\times \St_{S_k}$ to $G$. The
case $b(G)=\chi(1)b(G/N)=q^kb(G/N)$ can be argued as before by using
Lemma~\ref{lemma-N-G/N}. So we may assume that $b(G)>q^kb(G/N)$.
Equivalently, every irreducible character of $G$ lying above
$\St_{S_1}\times\St_{S_2}\times\cdots\times \St_{S_k}\in \Irr(N)$
has degree smaller than $b(G)$. It follows in particular that
\begin{equation}\label{equation 3}b(G)e=|G|-b(G)^2\geq
q^{2k}|G/N|.
\end{equation}

Let $M:=S_2\times\cdots\times S_k$. Let $T: = \bN_G (M)$, so $|G:T|
= k$. Furthermore $M$ can be considered as a subgroup of
$T/\bC_G(M)$, which in turn is isomorphic to a subgroup of
$\Aut(M)\cong\Aut(\PSL_2(q))\wr \Sy_{k-1}$. Using
\cite[Lemma~1.3]{Mattarei}, we have that
$\St_{S_2}\times\cdots\times \St_{S_k}\in\Irr(M)$ is extendible to
$\Aut(M)$, and hence is extendible to $T/\bC_G(M)$. It follows that
$\St_{S_2}\times\cdots\times \St_{S_k}\in\Irr(M)$ is extended to an
irreducible character of $T$ whose kernel contains $\bC_G(M)$. Now
since $S_1\subseteq \bC_G(M)$, we conclude that the character
$1_{S_1}\times \St_{S_2}\times....\times \St_{S_k}\in\Irr(N)$ is
extendible to $T$. Assume that $\chi_1$ is an extension of
$1_{S_1}\times \St_{S_2}\times....\times \St_{S_k}$ to $T$.

Observe that the stabilizer of $1_{S_1}\times
\St_{S_2}\times....\times \St_{S_k}$ normalizes $M$, and
$1_{S_1}\times \St_{S_2}\times....\times \St_{S_k}$ has exactly $k$
conjugates under the action of $G$. Thus, $T$ must be the stabilizer
of $1_{S_1}\times \St_{S_2}\times....\times \St_{S_k}$ in $G$.

Now we apply Gallagher's theorem to obtain a bijection
$\beta_1\mapsto \beta_1\chi_1$ between $\Irr(T/N)$ and the set of
irreducible characters of $T$ lying above $1_{S_1}\times
\St_{S_2}\times....\times \St_{S_k}\in \Irr(N)$. Moreover, by
Clifford's theorem, each irreducible character of $T$ lying above
$1_{S_1}\times \St_{S_2}\times....\times \St_{S_k}\in \Irr(N)$
induces irreducibly to $G$. Therefore, the map $\beta_1\mapsto
(\beta_1\chi_1)^G$ is a bijection between $\Irr(T/N)$ and the set of
irreducible characters of $G$ lying above $1_{S_1}\times
\St_{S_2}\times....\times \St_{S_k}\in \Irr(N)$. We note that
\[(\beta_1\chi_1)^G(1)=|G:T|\chi_1(1)\beta_1(1)=kq^{k-1}\beta_1(1)\] and
\[kq^{k-1}\beta_1(1)\leq kq^{k-1}b(T/N)\leq kq^{k-1}|T/N|^{1/2}=k^{1/2}q^{k-1}|G/N|^{1/2}.\]

If $b(G)=kq^{k-1}b(T/N)$ then it follows that
\[b(G)^{3/2}+b(G)\leq k^{3/2}q^{3(k-1)/2}|G/N|^3/4+k^{1/2}q^{k-1}|G/N|^{1/2}.\]
Using the hypothesis $|G/N|\geq q^k$, one can easily check that
\[k^{3/2}q^{3(k-1)/2}|G/N|^3/4+k^{1/2}q^{k-1}|G/N|^{1/2}<q^{2k}|G/N|\]
and therefore we have
\[b(G)^{3/2}+b(G)<q^{2k}|G/N|.\] This and (\ref{equation 3}) imply
that $b(G)^{3/2}+b(G)<b(G)e$. As in the proof of
Theorem~\ref{theorem-S-ncongPSL(2,q)}, we deduce that $|G|<e^4-e^3$
as required.

So from now on to the end of the proof we assume that
$b(G)>kq^{k-1}b(T/N)$. In other words, the irreducible characters of
$G$ of the form $(\beta_1\chi_1)^G$ where $\beta_1\in\Irr(T/N)$ all
have degree smaller than $b(G)$. Recall from the second paragraph
that all irreducible characters of $G$ lying above
$\St_{S_1}\times\cdots\times \St_{S_k}$ also have degree smaller
than $b(G)$. Therefore we obtain
\begin{align*}b(G)e&\geq \sum_{\beta\in\Irr(G/N)}\beta(1)^2\chi(1)^2+\sum_{\beta_1\in\Irr(T/N)}((\beta_1\chi_1)^G(1))^2\\
&= q^{2k}|G/N|+k^2q^{2(k-1)}|T/N|\\
&=q^{2k}|G/N|+kq^{2(k-1)}|G/N|.
\end{align*}

Using the hypothesis that $|G/N|\geq q^k$ and the fact that
$|N|=|\PSL_2(q)|^k<q^{3k}$, we easily check that
\[q^{2k}|G/N|>|G|^{3/4}\] and
\[kq^{2(k-1)}|G/N|>|G|^{1/2}.\] Therefore we deduce that
$b(G)e>|G|^{3/4}+|G|^{1/2}$. Since $b(G)\leq |G|^{1/2}$, it follows
that $b(G)e>b(G)^{3/2}+b(G)$ and the theorem follows as before.
\end{proof}

Unlike the groups in even characteristic, $\PSL_2(q)$ with odd $q$
may have the Steinberg character as the only one that is extendible
to $\Aut(\PSL_2(q)).$ According to \cite[p.~8]{White}, when $q$ is
odd, $\PSL_2(q)$ has the following irreducible characters:
\begin{enumerate}
\item[(i)] $1_{\PSL_2(q)}$ of degree 1,
\item[(ii)] $\St_{\PSL_2(q)}$ of degree $q$,
\item[(iii)] $\chi_i, 1\leq i\leq (q-3)/2$ and $i$ even, of degree $q+1$,
\item[(iv)] $\theta_j, 1\leq j\leq (q-1)/2$ and $j$ even, of degree
$q-1$,
\item[(v)] $\xi_1$ and $\xi_2$ of degree $(q+1)/2$, if $q\equiv
1(\bmod~4)$, and
\item[(vi)] $\eta_1$ and $\eta_2$ of degree $(q-1)/2$, if $q\equiv
-1(\bmod~4)$.
\end{enumerate}

Let $q=p^f$ where $p$ is an odd prime. Let $\varphi$ be the field
automorphism of order $f$ of $\PSL_2(q)$ and $\delta$ be the
diagonal automorphism of order $2$ of $\PSL_2(q)$. Then, by
\cite[Lemma~4.8]{White}, the character $\chi_i\in\Irr(\PSL_2(q))$ is
invariant under $\varphi^k$ where $1\leq k\leq f$ if and only if
$(p^f-1)\mid i(p^k-1)$ or $(p^f-1)\mid i(p^k+1)$; and the character
$\theta_j\in\Irr(\PSL_2(q))$ is invariant under $\varphi^k$ if and
only if $(p^f+1)\mid j(p^k-1)$ or $(p^f+1)\mid j(p^k+1)$. Contrary
to the even characteristic case, we now show that $\PSL_2(p^f)$ has
an irreducible character of degree $q-1$ whose stabilizer in
$\Aut(\PSL_2(q))$ is $\PGL_2(q)$, which is as small as possible.

\begin{lemma}\label{lemma-Stab}
Let $q=p^f\geq 5$ be an odd prime power and let $\theta_2$ be
defined as above. Then
\[\Stab_{\Aut(\PSL_2(q))}(\theta_2)=\PGL_2(q).\]
\end{lemma}

\begin{proof}
We observe that $(p^f+1)\mid 2(p^k-1)$ or $(p^f+1)\mid 2(p^k+1)$ if
and only if $k=f$. That means $\theta_2\in\Irr(\PSL_2(q))$ is not
invariant under $\varphi^k$ for every $1\leq k<f$. It is well known
that every irreducible character of $\PSL_2(q)$ of degree $q\pm 1$
is invariant under the diagonal automorphism $\delta$. Therefore
\[\Stab_{\Aut(\PSL_2(q))}(\theta_2)=\PSL_2(q)\rtimes \langle\delta
\rangle=\PGL_2(q),\] as claimed.
\end{proof}

\begin{theorem}\label{theorem-S-congPSL(2,q)2}
Assume that $N=\PSL_2(q)\times\cdots\times \PSL_2(q)$, a direct
product of $k$ copies of $\PSL_2(q)$ with $q\geq5$, is a minimal
normal subgroup of a finite group $G$ such that $|G/N|< q^k$. Let
$|G|=b(G)(b(G)+e)$. Then $e>\sqrt{b(G)}+1$ and, in particular, $|G|<
e^4-e^3$.
\end{theorem}

\begin{proof}
Arguing as in the proof of Theorem~\ref{theorem-S-congPSL(2,q)1}, we
can assume that every irreducible character of $G$ lying above
$\St_{S_1}\times\St_{S_2}\times\cdots\times \St_{S_k}\in \Irr(N)$
has degree smaller than $b(G)$.

By Lemma~\ref{lemma-Stab}, we have
$\Stab_{\Aut(\PSL_2(q))}(\theta_2)=\PGL_2(q)$. Let
$\psi:=\theta_2\times\cdots\times \theta_2\in\Irr(N)$. Then we have
$\Stab_{\Aut(N)}(\psi)=\PGL_2(q)\wr \Sy_{k}$. Set
$\overline{H}:=\PGL_2(q)\wr \Sy_{k}$.

Consider $N$ as a subgroup of $G/\bC_G(N)$, which in turn can be
considered as a subgroup of $\Aut(N)$. Then the stabilizer of $\psi$
in $G/\bC_G(N)$ is $\overline{H}\cap G/\bC_G(N)$. Let $H$ be the
preimage of $\overline{H}\cap G/\bC_G(N)$ in $G$. Then we have
$\Stab_G(\psi)=H$.

Recall that $\PGL_2(q)=\PSL_2(q)\rtimes \langle \delta\rangle$ where
$\delta$ the diagonal automorphism of degree $2$ of $\PSL_2(q)$.
Therefore $\theta$ is extendible to $\PGL_2(q)$. Thus
$\psi\in\Irr(N)$ is extendible to $\overline{H}$ so that it is
extendible to $\overline{H}\cap G/\bC_G(N)$ as well. We deduce that
$\psi$ is extendible to $H$. Let $\chi$ be an extension of $\psi$ to
$H$.

The conclusions of the last two paragraphs, together with
Gallagher's theorem and Clifford's theorem, imply that $\beta\mapsto
(\beta\chi)^G$ is a bijection between $\Irr(H/N)$ and the set of
irreducible characters of $G$ lying above $\psi\in \Irr(N)$. Note
that
\[
(\beta\chi)^G(1)=\beta(1)\chi(1)|G/H|=(q-1)^k\beta(1)|G/H|.
\]

We come up with two cases:

\medskip

\textbf{Case} $b(G)=(q-1)^kb(H/N)|G/H|$: Then we have $b(G)\leq
(q-1)^|G/N|$. Recall that every irreducible character of $G$ lying
above $\St_{S_1}\times\St_{S_2}\times\cdots\times \St_{S_k}\in
\Irr(N)$ has degree smaller than $b(G)$. Therefore $b(G)e\geq
q^{2k}|G/N|$. This and the inequality $b(G)\leq (q-1)^|G/N|$,
together with the hypothesis that $|G/N|<q^k$ imply that
$b(G)e>b(G)^{3/2}+b(G)$, and we are done as before.

\medskip

\textbf{Case} $b(G)>(q-1)^kb(H/N)|G/H|$: Then every irreducible
character of $G$ of the form $(\beta\chi)^G$ where
$\beta\in\Irr(H/N)$ has degree smaller than $b(G)$. Therefore
\begin{align*}b(G)e&\geq q^{2k}|G/N|+\sum_{\beta\in\Irr(H/N)}((\beta\chi)^G(1))^2\\
&= q^{2k}|G/N|+(q-1)^{2k}|H/N||G/H|^2\\
&\geq q^{2k}|G/N|+(q-1)^{2k}|G/N|.
\end{align*}

Using $|G/N|<q^k$, we can check that
\[q^{2k}|G/N|+(q-1)^{2k}|G/N|> (q+1)^{3k/2}|G/N|^{3/2}+(q+1)^{k}|G/N|.\]
It follows from Lemma~\ref{lemma-b(G)} that
\[q^{2k}|G/N|+(q-1)^{2k}|G/N|> b(G)^{3/2}+b(G).\]
This and the above inequality $b(G)e\geq
q^{2k}|G/N|+(q-1)^{2k}|G/N|$ imply that $b(G)e> b(G)^{3/2}+b(G)$,
which in turn implies that $b(G)<e^2-e$ and the theorem follows.
\end{proof}

Theorems~\ref{theorem-main} and \ref{theorem-main-3} now are
consequences of Theorems~\ref{theorem-S-ncongPSL(2,q)},
\ref{theorem-S-congPSL(2,q)q-even}, \ref{theorem-S-congPSL(2,q)1},
and~\ref{theorem-S-congPSL(2,q)2}.


\section{Groups with $|G| = e^4 - e^3$}\label{section: final}

In this section, we characterize those groups that satisfy the
condition $|G| = e^4 - e^3$.  To do this, we need to introduce
another class of groups.

An irreducible character $\chi$ of a finite group $G$ is said to be
a \emph{Gagola character} if it vanishes on all but two conjugacy
classes of $G$. Groups with such a character have been studied in
great depth by Gagola in \cite{Gagola}. In particular, if $G$ has a
Gagola character, then $G$ has a unique minimal normal subgroup $N$,
which is necessarily elementary abelian. Furthermore, $\chi$
vanishes on all the elements in $G \setminus N$ and that $\chi$ is
the unique irreducible character of $G$ whose kernel does not
contain $N$. In this situation, for simplicity we will say that $G$
is a \emph{Gagola group} and $(G,N)$ is a \emph{Gagola pair}.

The following lemma shows the connection between groups in
consideration and Gagola groups.

\begin{lemma}\label{lemma-Lewis}
\begin{enumerate}
\item[(i)] Let $G$ be a finite group with a nontrivial abelian normal subgroup, and
let $|G| = d(d+e)$ where $d$ is a character degree of $G$ and $e
> 1$ is an integer. If $d\geq e^2-e$ then $G$ has a Gagola character
$\chi\in\Irr(G)$ of degree $d$.

\item[(ii)] Let $(G,N)$ be a Gagola pair with the associated Gagola character of degree $d$. Let $p$ be the only prime
divisor of $|N|$ and $P$ a Sylow $p$-subgroup of $G$. Then
$|P:N|=e^2$ and $d=e(|N|-1)$.
\end{enumerate}
\end{lemma}

\begin{proof}
This is \cite[Lemmas~2.1 and~2.2]{Lewis}.
\end{proof}

We can now characterize the groups $G$ with $|G| = e^4 - e^3$.

\begin{theorem} \label{equal}
Let $G$ be a finite group, and let $|G| = d(d+e)$ where $d > 1$ is a
character degree of $G$ and $e
> 1$ is an integer.  Then $|G| = e^4 - e^3$ if and only if $G$ has a
Gagola character of degree $d$ and a unique minimal normal subgroup
$N$ of order $e$.
\end{theorem}

\begin{proof}
Suppose first that $G$ has a Gagola character of degree $d$ and the
unique minimal normal subgroup $N$ with $|N| = e$.  Let $p$ be the
unique prime divisor of $|N|$ and let $P$ be a Sylow $p$-subgroup of
$G$.  By Lemma~\ref{lemma-Lewis}(ii), we know that $e^2 = |P:N|$.
Furthermore, from Lemma~2.1 and Corollary~2.3 of \cite{Gagola}, we
have $|G:P| = |N| - 1$. Therefore,
$$
|G| = |G:P||P:N||N| = (|N| - 1)|P:N| |N| = (e-1)e^2 e = e^4 - e^3.
$$

Conversely, suppose that $|G| = e^4 - e^3$.  In view of
Theorem~\ref{theorem-main}, $G$ must have a nontrivial solvable
radical. In particular, $G$ has a nontrivial abelian normal
subgroup. Theorem~1.1 of \cite{Lewis} then implies that \[d \le e^2
- e.\] If $d < e^2 - e$, then \[|G| = d(d+e) < (e^2 - e)((e^2-e)+e)
= (e^2-e)e^2 = e^4 - e^3 = |G|,\] which is a contradiction.  Thus,
we must have $d = e^2 - e$. We then apply Lemma~\ref{lemma-Lewis}(i)
to see that $G$ has a Gagola character of degree $d$, and hence has
a unique minimal normal subgroup. Let $N$ be the unique minimal
normal subgroup of $G$. Applying Lemma~\ref{lemma-Lewis}(ii), we
deduce that $d = e (|N| - 1)$. Since $d = e^2 - e$, it follows that
$e (e-1) = e (|N| - 1)$, and we easily computes that $|N| = e$.
\end{proof}

The groups mentioned in the introduction are not the only Gagola
groups in the consideration of Theorem~\ref{equal}. Let us describe
here another family of such groups, which appeared in
\cite[p.~409]{Goldstein-Guralnick-Lewis-Moreto-Navarro-Tiep} in a
different context. These groups have normal Sylow $p$-subgroups,
where $p$ the the prime divisor of $|N|$.

Let $\FF$ be a field of order $q$ where $q$ is a power of some prime
$p$. Take
\[\displaystyle K := \left\{ \begin{pmatrix} 1 & a & b \\
                       0 & 1 & c \\
                       0 & 0 & d  \end{pmatrix} : a, b, c \in \FF; d \in \FF^*
                       \right\}.\]
Let ${\mathcal G}:=\Gal(\FF/\FF_p)$ be the Galois group for $\FF$
over the subfield $\FF_p$ of order $p$. We define an action
${\mathcal G}$ on $K$ as follows: if $\sigma \in {\mathcal G}$, then
$\sigma$ acts on a typical element of $K$ by acting on each of the
entries of $K$. Let
\[\displaystyle P :=\left\{
 \begin{pmatrix}  1 & a & b \\
                  0 & 1 & c \\
                  0 & 0 & 1  \end{pmatrix} \mid a,b,c \in \FF
                  \right\},\]
and \[\displaystyle L:=\left\{
 \begin{pmatrix}  1 & 0 & 0 \\
                  0 & 1 & 0 \\
                  0 & 0 & d  \end{pmatrix} \mid d \in \FF^*
                  \right\}.\]
It is not difficult to see that $P$ is an ultraspecial group of
order $q^3$ and $L$ is a cyclic group of order $q - 1$.  Notice that
$P$ and $L$ are invariant under the action of ${\mathcal G}$.
Furthermore, the semi-direct product of ${\mathcal G}$ acting on $L$
is isomorphic to the affine group on $\FF$. Let $\Gamma$ be the
semi-direct product of ${\mathcal G}$ acting on $K$. (We note that
$\bZ(P) L {\mathcal G}$ is isomorphic to the affine semi-linear
group on $\FF$, which has been discussed on
\cite[p.~38]{Manz-Wolf}.)

Suppose $D = NH^*$ is a two-transitive Frobenius group of Dickson
type of order $p^n (p^n - 1)$, where $N$ is the Frobenius kernel and
$H^*$ is the Frobenius complement. It is well-known that $H^*$ can
be embedded in the affine group of $\FF$ and that $NH^*$ is
isomorphic to a subgroup of the semi-linear affine group of $\FF$.
Thus, $H^*$ is isomorphic to $H \subseteq L{\mathcal G}\subset
\Gamma$ and $NH$ is isomorphic to $\bZ(P) H$.  We set $G: = PH$, and
it is not difficult to see that $G$ is a Gagola group with the
desired properties.

A family of non-$p$-closed examples can be found in
\cite[Theorem~3.3]{Lewis2} for every prime $p$. These groups were
constructed as subgroups of index $p$ of the group $\Gamma$ defined
above when $q = p^p$. Two other non-$p$-closed examples can be found
in \cite[pp.~383-384]{Gagola}.  The first of these has a subgroup
$S$ of order $12$ obtained by taking a cyclic group of order $4$
acting on a group of order $3$ by inverting the nontrivial elements
and then having $S$ act on the direct product of two cyclic groups
of order $4$.  The second one has a subgroup $T$ which is the direct
product of a cyclic group of order $4$ and the semi-direct product
of a cyclic group of order $9$ acting nontrivially on the quaternion
group of order $8$. The desired group is then obtained by having $T$
act on the direct product of two cyclic groups of order $9$.  We
refer the interested reader to \cite{Gagola} for detailed
constructions of these groups.

It seems nontrivial to us to obtain a complete classification of
those groups that satisfy the extremal condition $|G|=e^4-e^3$. It
is likely that these groups are necessarily solvable, but we are not
able to confirm it at this time.



\end{document}